\documentclass{amsart}
\usepackage{amssymb}
\usepackage[mathcal]{eucal}
\usepackage{upref}

\usepackage[lite,initials,shortalphabetic]{amsrefs}

\usepackage[all,cmtip]{xy}

\DeclareMathAlphabet{\EuRm}{U}{eur}{m}{n}
\SetMathAlphabet{\EuRm}{bold}{U}{eur}{b}{n}

\providecommand{\texorpdfstring}[2]{#1}

\newtheoremstyle{slplain}
 {.5\baselineskip\@plus.2\baselineskip\@minus.2\baselineskip}
 {.5\baselineskip\@plus.2\baselineskip\@minus.2\baselineskip}
 {\slshape}
 {}
 {\bfseries}
 {.}
 { }
 {}

\numberwithin{equation}{section}
\theoremstyle{slplain}
\newtheorem{thm}[equation]{Theorem}  

\newtheorem{prop}[equation]{Proposition}  

\theoremstyle{definition}
\newtheorem{defn}[equation]{Definition} 

\theoremstyle{remark}
\newtheorem{notation}[equation]{Notation}

\newcommand{\thmref}{Theorem~\ref}
\newcommand{\propref}{Proposition~\ref}

\newcommand{\defref}{Definition~\ref}

\newcommand{\diagref}{Diagram~\ref}

\newcommand{\secref}{Section~\ref}

\newcommand{\bdry}{\partial}
\newcommand{\phitilde}{\widetilde\phi}
\newcommand{\starhat}{\mathbin{\hat *}}
\newcommand{\startilde}{\mathbin{\tilde *}}

\newcommand{\limone}{\mathrm{lim}^{1}}

\DeclareMathOperator{\Ker}{Ker}
\newcommand{\homotopic}{\simeq}

\newcommand{\Comma}{\rlap{\enspace ,}}
\newcommand{\Period}{\rlap{\enspace .}}

\begin{document}

\title[The homotopy groups of the inverse limit of a tower of
      fibrations]{The homotopy groups of the inverse limit\\
      of a tower of fibrations}

\author{Philip S. Hirschhorn}

\address{Department of Mathematics\\
   Wellesley College\\
   Wellesley, Massachusetts 02481}

\email{psh@math.mit.edu}

\urladdr{http://www-math.mit.edu/\textasciitilde psh}

\date{May 16, 2015}

\begin{abstract}
  We carefully present an elementary proof of the well known theorem
  that each homotopy group (or, in degree zero, pointed set) of the
  inverse limit of a tower of fibrations maps naturally onto the
  inverse limit of the homotopy groups (or, in degree zero, pointed
  sets) of the spaces in the tower, with kernel naturally isomorphic
  to $\limone$ of the tower of homotopy groups of one dimension
  higher.
\end{abstract}

\maketitle

The main theorem here is \thmref{thm:HomGrps}.  This is due to Gray
\cite{Gray}, Quillen \cite{RHT}*{Proposition~3.8}, Vogt \cite{Vogt},
Cohen \cites{Cohen-aarhus,Cohen}, and Bousfield and Kan
\cite{YM}*{Theorem 3.1 in Chapter IX, section 3}, some of which only
consider the case of simply connected spaces or spaces with abelian
fundamental groups.  The proof given here is at least morally the one
in \cite{Cohen}.

\section{\texorpdfstring{$\limone$}{lim1} of a tower of not necessarily
  abelian groups}
\label{sec:lim1}

This definition is as in \cite{Vogt} and \cite{YM}*{Chapter IX,
  section 2}.

\begin{defn}
  \label{def:lim1}
  If
  \begin{displaymath}
    \xymatrix{
      {\cdots} \ar[r]
      & {G_{n+1}} \ar[r]^{p_{n+1}}
      & {G_{n}} \ar[r]^{p_{n}}
      & {G_{n-1}} \ar[r]^{p_{n-1}}
      & {\cdots} \ar[r]^{p_{1}}
      &{G_{0}}
    }
  \end{displaymath}
  is a tower of groups, then we define a left action of the product
  group $\prod_{n\ge 0} G_{i}$ on the product set $\prod_{n\ge 0}
  G_{i}$ by letting the action of an element $(g_{n})_{n\ge 0} =
  (g_{0}, g_{1}, g_{2}, \ldots)$ on an element $(h_{n})_{n \ge 0} =
  (h_{0}, h_{1}, h_{2}, \ldots)$ be
  \begin{displaymath}
    \bigl(g_{n}h_{n}\bigl(p_{n+1}(g_{n+1}^{-1})\bigr)\bigr)_{n\ge 0}=
    \bigl(g_{0}h_{0}\bigl(p_{1}(g_{1}^{-1})\bigr),
      g_{1}h_{1}\bigl(p_{2}(g_{2}^{-1})\bigr),
      g_{2}h_{2}\bigl(p_{3}(g_{3}^{-1})\bigr),
      \ldots
    \bigr) \Period
  \end{displaymath}
  We define $\limone$ of that tower of groups to be the pointed set
  that is the set of orbits of this action.  If all of the groups
  $G_{n}$ are abelian, then $\limone$ of the tower is isomorphic to
  the underlying pointed set of the cokernel of the homomorphism
  $f\colon \prod_{n\ge 0} G_{n} \to \prod_{n\ge 0} G_{n}$ defined by
  \begin{displaymath}
    f(g_{0}, g_{1}, g_{2},\ldots) =
    \bigl(g_{0}\bigl(p_{1}(g_{1}^{-1})\bigr),
    g_{1}\bigl(p_{2}(g_{2}^{-1})\bigr),
    g_{2}\bigl(p_{3}(g_{3}^{-1})\bigr),\ldots\bigr) \Comma 
  \end{displaymath}
  and thus has a natural abelian group structure.
\end{defn}

\section{The main theorems}
\label{sec:MainThm}

\begin{thm}
  \label{thm:HomGrps}
  Let
  \begin{displaymath}
    \xymatrix{
      {\cdots} \ar[r]
      & {X_{n+1}} \ar[r]^{p_{n+1}}
      & {X_{n}} \ar[r]^{p_{n}}
      & {X_{n-1}} \ar[r]^{p_{n-1}}
      & {\cdots} \ar[r]^{p_{1}}
      & {X_{0}}
    }
  \end{displaymath}
  be a tower of fibrations of pointed spaces.  For every choice of
  basepoint in $\lim_{n}X_{n}$ (whose image is then taken as the
  basepoint in each $X_{n}$ for $n \ge 0$) and every $k \ge 0$ there
  is a natural short exact sequence
  \begin{displaymath}
    \xymatrix{
      {1} \ar[r]
      & {\limone_{n} \pi_{k+1} X_{n}} \ar[r]
      & {\pi_{k}\lim_{n} X_{n}} \ar[r]^{P}
      & {\lim_{n} \pi_{k}X_{n}} \ar[r]
      & {1}
    }
  \end{displaymath}
  of abelian groups if $k \ge 2$, of groups if $k = 1$, and of pointed
  sets if $k = 0$.  If $k = 0$, all of the spaces $X_{n}$ are
  H-spaces, and all of the maps $p_{n}$ are H-maps, then this is a
  short exact sequence of groups.
\end{thm}
The proof of \thmref{thm:HomGrps} is in \secref{sec:prfmain}.

\begin{thm}
  \label{thm:limwe}
  If
  \begin{displaymath}
    \xymatrix{
      {\cdots} \ar[r]
      & {X_{n+1}} \ar[r]^-{p_{n+1}} \ar[d]^{f_{n+1}}
      & {X_{n}} \ar[r]^-{p_{n}} \ar[d]^{f_{n}}
      & {X_{n-1}} \ar[r]^-{p_{n-1}} \ar[d]^{f_{n-1}}
      & {\cdots} \ar[r]^-{p_{1}}
      & {X_{0}} \ar[d]^{f_{0}}\\
      {\cdots} \ar[r]
      & {Y_{n+1}} \ar[r]_-{q_{n+1}}
      & {Y_{n}} \ar[r]_-{q_{n}}
      & {Y_{n-1}} \ar[r]_-{q_{n-1}}
      & {\cdots} \ar[r]_-{q_{1}}
      & {Y_{0}}
    }
  \end{displaymath}
  is a map of towers of fibrations of topological spaces such that
  $f_{n}$ is a weak equivalence for every $n \ge 0$, then the induced
  map of limits $f\colon \lim_{n} X_{n} \to \lim_{n} Y_{n}$ is a weak
  equivalence.
\end{thm}

\begin{proof}
  This theorem mostly follows from \thmref{thm:HomGrps} and the five
  lemma, but care must be taken to show that the set of path
  components is mapped isomorphically.  We must first show that
  $\lim_{n}X_{n}$ is nonempty if and only if $\lim_{n}Y_{n}$ is
  nonempty.  If $\lim_{n}X_{n}$ is nonempty, then the existence of the
  natural map $\lim_{n}X_{n} \to \lim_{n}Y_{n}$ implies that
  $\lim_{n}Y_{n}$ is nonempty.

  Conversely, if $\lim_{n}Y_{n}$ is nonempty, then we can choose a
  point $(y_{n})_{n \ge 0}$ in $\lim_{n}Y_{n}$, and for each $n \ge 0$
  we can choose a point $x_{n} \in X_{n}$ such that $f_{n}(x_{n})$ is
  in the same path component as $y_{n}$.  We will now inductively
  define points $x'_{n} \in X_{n}$ such that $x'_{n}$ is in the same
  path component as $x_{n}$ and $p_{n+1}(x'_{n+1}) = x'_{n}$ for all
  $n \ge 0$.  We begin the induction by letting $x'_{0} = x_{0}$.  If
  $n \ge 0$ and we've defined $x'_{n}$, then $q_{n+1}f_{n+1}(x_{n+1})
  = f_{n}p_{n+1}(x_{n+1})$ is in the same path component as
  $q_{n+1}(y_{n+1}) = y_{n}$, and so (since $f_{n}$ is a weak
  equivalence) $p_{n+1}(x_{n+1})$ is in the same path component as
  $x'_{n}$.  Thus, we can choose a path $\alpha\colon I \to X_{n}$
  from $p_{n+1}(x_{n+1})$ to $x'_{n}$ and then lift $\alpha$ to a path
  $\widetilde\alpha\colon I \to X_{n+1}$ such that
  $\widetilde\alpha(0) = x_{n+1}$; we let $x'_{n+1} =
  \widetilde\alpha(1)$.  This completes the induction, and so
  $\lim_{n}X_{n}$ contains the point $(x'_{n})_{n\ge 0}$.

  In the case in which $\lim_{n}X_{n}$ and $\lim_{n}Y_{n}$ are
  nonempty, for every choice of basepoint in $\lim_{n}X_{n}$ (whose
  image is then taken as the basepoint in each of $\lim_{n}Y_{n}$,
  $X_{n}$, and $Y_{n}$ for $n \ge 0$) and every $k \ge 0$
  \thmref{thm:HomGrps} implies that we have a map of short exact
  sequences
  \begin{equation}
    \label{eq:ses}
    \vcenter{
      \xymatrix{
        {1} \ar[r]
        & {\limone_{n} \pi_{k+1} X_{n}} \ar[r] \ar[d]_{\phi}
        & {\pi_{k}\lim_{n}X_{n}} \ar[r]^{P} \ar[d]_{f_{*}}
        & {\lim_{n}\pi_{k}X_{n}} \ar[r] \ar[d]^{\psi}
        & {1}\\
        {1} \ar[r]
        & {\limone_{n} \pi_{k+1} Y_{n}} \ar[r]
        & {\pi_{k}\lim_{n}Y_{n}} \ar[r]^{Q}
        & {\lim_{n}\pi_{k}Y_{n}} \ar[r]
        & {1}
      }
    }
  \end{equation}
  of groups if $k>0$ and of pointed sets if $k=0$, and both $\phi$ and
  $\psi$ are isomorphisms.  For $k > 0$ the five lemma (non-abelian,
  if $k=1$) applied to \diagref{eq:ses} shows that for every choice of
  basepoint in $\lim_{n}X_{n}$ the map $\pi_{k}\lim_{n}X_{n} \to
  \pi_{k}\lim_{n}Y_{n}$ is an isomorphism, and so it remains only to
  show that the set of path components of $\lim_{n}X_{n}$ maps
  isomorphically to the set of path components of $\lim_{n}Y_{n}$.

  To see that the map of path components is surjective, choose some
  basepoint for $\lim_{n}X_{n}$ and consider \diagref{eq:ses}.  Let $a
  \in \pi_{0}\lim_{n}Y_{n}$; then we can choose $b \in
  \pi_{0}\lim_{n}X_{n}$ such that $P(b) = \psi^{-1}Q(a)$, and we will
  have that $Q\bigl(f_{*}(b)\bigr) = Q(a)$.  Now choose a new
  basepoint for $\lim_{n}X_{n}$ that is in the path component $b$ of
  $\lim_{n}X_{n}$, and consider this version of \diagref{eq:ses}.  The
  path component $f_{*}(b)$ of $\lim_{n}Y_{n}$ is now the path
  component of the basepoint, and so $a \in {\limone_{n} \pi_{k+1}
    Y_{n}}$.  Thus, there is an element $a' \in {\limone_{n} \pi_{k+1}
    X_{n}}$ such that $\phi(a') = a$, and $a'$ is an element of
  $\pi_{0}\lim_{n}X_{n}$ that goes to $a$ under $f_{*}$.

  To see that the map of path components is injective, let $a$ and $b$
  be path components of $\lim_{n}X_{n}$ that go to the same path
  component of $\lim_{n}Y_{n}$.  Choose a basepoint in the path
  component $a$, and consider \diagref{eq:ses}.  We have $f_{*}(a) =
  f_{*}(b)$, and so $P(a) = P(b)$, and since $a$ is the path component
  of the basepoint both $a$ and $b$ are elements of ${\limone_{n}
    \pi_{k+1} X_{n}}$.  If $a \neq b$, then $\phi(a) \neq \phi(b)$,
  and so $\phi(a)$ and $\phi(b)$ are distinct elements of
  $\pi_{0}\lim_{n}Y_{n}$.  Since those are the same element of
  $\pi_{0}\lim_{n}Y_{n}$, it must be that $a = b$ in $\lim_{n}X_{n}$.
\end{proof}

\section{Homotopy groups}
\label{sec:HtpyGrps}

\begin{notation}
  \label{not:cubes}
  We let $I$ denote the interval $[0,1]$, and we let $I^{0}$ denote a
  single point.  If $k \ge 1$ we will often denote a point of $I^{k}$
  as $(p,t)$, where $p \in I^{k-1}$ and $t \in I$.
\end{notation}

\begin{defn}
  \label{def:mapcubes}
  If $X$ is a space and $k \ge 0$, we will represent elements of
  $\pi_{k}X$ by maps $\alpha\colon I^{k} \to X$ that take the boundary
  of $I^{k}$ to the basepoint, and we will denote the element of
  $\pi_{k}X$ represented by $\alpha$ as $[\alpha]$.

  We will almost always multiply and take inverses of elements of
  $\pi_{k}X$ using the last coordinate of $I^{k}$.
  \begin{enumerate}
  \item If $k \ge 1$ and $\alpha\colon I^{k} \to X$ is a map, then the
    \emph{inverse} $\alpha^{-1}\colon I^{k} \to X$ of $\alpha$ is the
    map defined by $\alpha^{-1}(p,t) = \alpha(p, 1-t)$.  If $k = 1$
    this is the usual definition of the inverse path.
  \item If $k \ge 1$ and $\alpha,\beta\colon I^{k} \to X$ are maps
    such that $\alpha(p,1)= \beta(p,0)$ for all $p \in I^{k-1}$, then
    the \emph{composition} of $\alpha$ and $\beta$ is the map
    $\alpha*\beta\colon I^{k} \to X$ defined by
    \begin{displaymath}
      (\alpha*\beta)(p,t) =
      \begin{cases}
        \alpha(p, 2t)& \text{if $0 \le t \le \frac{1}{2}$}\\
        \beta(p, 2t-1)& \text{if $\frac{1}{2} \le t \le 1$}\Period
      \end{cases}
    \end{displaymath}
    If $k = 1$, this is the usual definition of the composition of
    paths.
  \end{enumerate}
\end{defn}

\begin{notation}
  \label{not:homotopy}
  If $k \ge 0$ and $\alpha,\beta\colon I^{k} \to X$ are maps such that
  $\alpha(p) = \beta(p)$ for all $p \in \bdry(I^{k})$, then $\alpha
  \homotopic \beta$ will mean that $\alpha$ is homotopic to $\beta$
  relative to the boundary of $I^{k}$.  If $k = 1$, this is the usual
  notion of path homotopy.
\end{notation}

\begin{defn}
  \label{def:nullhomotopy}
  If $k \ge 0$ and $\alpha\colon I^{k} \to X$ is a map taking the
  boundary of $I^{k}$ to the basepoint of $X$, then by a
  \emph{nullhomotopy} of $\alpha$ we will mean a map $\beta\colon
  I^{k+1} \to X$ such that
  \begin{alignat*}{2}
    \beta(p,0) &= *&\quad& \text{for all $p \in I^{k}$}\\
    \beta(p,1) &= \alpha(p)&& \text{for all $p \in I^{k}$, and}\\
    \beta(p,t) &= *&& \text{for all $p \in \bdry I^{k}$} \Period
  \end{alignat*}
  There exists such a nullhomotopy for $\alpha$ if and only if
  $[\alpha] = 0$ in $\pi_{k}X$.
\end{defn}

We will occasionally need to compose maps of cubes using the
\emph{next to last coordinate}.
\begin{defn}
  \label{def:altcomp}
  If $k \ge 1$ and we have maps $\beta,\bar\beta\colon I^{k+1} \to X$
  such that
  \begin{displaymath}
    \beta(t_{1},t_{2},t_{k-1},1,t_{k+1}) =
    \bar\beta(t_{1},t_{2},t_{k-1},0,t_{k+1})
  \end{displaymath}
  then we define
  $\beta_{n}\starhat \bar\beta_{n}\colon I^{k+1} \to X_{n}$ by
  \begin{displaymath}
    (\beta_{n}\starhat\bar\beta_{n})(t_{1}, t_{2}, \ldots, t_{k+1})=
    \begin{cases}
      \beta_{n}(t_{1}, t_{2},\ldots, t_{k-1}, 2t_{k},t_{k+1})&
      \text{if $0 \le t_{k} \le \frac{1}{2}$}\\
      \bar\beta_{n}(t_{1}, t_{2},\ldots, t_{k-1}, 2t_{k}-1,t_{k+1})&
      \text{if $\frac{1}{2} \le t_{k} \le 1$}
    \end{cases}
  \end{displaymath}
\end{defn}
If both $\beta$ and $\bar\beta$ take the boundary of $I^{k+1}$ to the
basepoint of $X$, and thus represent elements of $\pi_{k+1}X$, then we
have $[\beta][\bar\beta] = [\beta*\bar\beta] =
[\beta\starhat\bar\beta]$ in $\pi_{k+1}X$.

\section{Proof of \thmref{thm:HomGrps}}
\label{sec:prfmain}

We prove that the natural map $P\colon \pi_{k}\lim_{n} X_{n} \to
\lim_{n} \pi_{k} X$ is surjective in \propref{prop:surj}, and in the
remainder of this section we discuss the kernel of $P$ and the case in
which the tower consists of H-maps between H-spaces.

\begin{notation}
  \label{not:projections}
  Given a tower of spaces as in \thmref{thm:HomGrps}, for each $i \ge
  0$ we let $P_{i}\colon \lim_{n} X_{n} \to X_{i}$ denote the natural
  projection from the limit.
\end{notation}

\begin{prop}
  \label{prop:surj}
  If $k \ge 0$, the natural map $P\colon \pi_{k} \lim_{n} X_{n} \to
  \lim_{n} \pi_{k}X_{n}$ of \thmref{thm:HomGrps} is surjective.
\end{prop}

\begin{proof}
  Let $([\alpha_{n}])_{n\ge 0}$ define an element of
  $\lim_{n}\pi_{k}X_{n}$, so that the $\alpha_{n}\colon I^{k} \to
  X_{n}$ for all $n \ge 0$ are maps taking the boundary of $I^{k}$ to
  the basepoint of $X_{n}$ and such that $p_{n+1}\circ\alpha_{n+1}
  \homotopic \alpha_{n}$.  We will inductively define maps
  $\bar\alpha_{n}\colon I^{k} \to X_{n}$ such that $\bar\alpha_{n}
  \homotopic \alpha_{n}$ and $p_{n+1}\circ\bar\alpha_{n+1} =
  \bar\alpha_{n}$; the collection $(\bar\alpha_{n})_{n\ge 0}$ will
  then define a map $\bar\alpha\colon I^{k} \to \lim_{n}X_{n}$ such
  that $P([\bar\alpha]) = ([\alpha_{n}])_{n\ge 0}$.

  We begin the induction by letting $\bar\alpha_{0} = \alpha_{0}$.  If
  $n \ge 0$ and we've defined $\bar\alpha_{n}$, then
  $p_{n+1}\circ\alpha_{n+1} \homotopic \bar\alpha_{n}$, and we can
  lift a homotopy to obtain a homotopy in $X_{n+1}$ that begins at
  $\alpha_{n+1}$ and ends at a map we'll call $\bar\alpha_{n+1}$ such
  that $p_{n+1}\circ\bar\alpha_{n+1} = \bar\alpha_{n}$.
\end{proof}

We turn now to the kernel of $P$.  Let $k \ge 0$, and let
$\alpha\colon I^{k} \to \lim_{n}X_{n}$ represent an element of
$\pi_{k}\lim_{n}X_{n}$ in the kernel of $P$.  For each $n \ge 0$ we
let $\alpha_{n} = P_{n}\circ \alpha\colon I^{k} \to X_{n}$ and we can
choose a nullhomotopy $\beta_{n}\colon I^{k+1} \to X_{n}$ of
$\alpha_{n}$ (see \defref{def:nullhomotopy}).  The map
$\beta_{n}*(p_{n+1}\circ \beta_{n+1}^{-1})\colon I^{k+1} \to X$ takes
the boundary of $I^{k+1}$ to the basepoint of $X_{n}$ and thus defines
an element $[\beta_{n}*(p_{n+1}\circ \beta_{n+1}^{-1})]$ of
$\pi_{k+1}X_{n}$, and these elements for all $n \ge 0$ define an
element $\phi(\alpha)$ of $\limone_{n}\pi_{k+1}X_{n}$.

\begin{prop}
  \label{prop:welldefined}
  The element $\phi(\alpha) \in \limone_{n} \pi_{k+1}X_{n}$ is
  independent of the choice of nullhomotopies $(\beta_{n})_{n \ge
    0}$.
\end{prop}

\begin{proof}
  If for each $n \ge 0$ we choose a different nullhomotopy
  $\bar\beta_{n}$, then the map $\bar\beta_{n}* \beta_{n}^{-1}\colon
  I^{k+1} \to X$ takes the boundary of $I^{k+1}$ to the basepoint of
  $X$ and thus defines an element $g_{n} =
  [\bar\beta_{n}*\beta_{n}^{-1}]$ of $\pi_{k+1}X_{n}$, and the element
  \begin{displaymath}
    g_{n}[\beta_{n}*(p_{n+1}\circ\beta_{n+1}^{-1})]
    \bigl((p_{n+1})_{*} g_{n+1}^{-1}\bigr)
  \end{displaymath}
  of $\pi_{k+1}X_{n}$ is represented by the map
  \begin{align*}
    &\mathrel{\phantom{\homotopic}}
    (\bar\beta_{n}*\beta_{n}^{-1})*
    \bigl(\beta_{n}*(p_{n+1}\circ\beta_{n+1}^{-1})\bigr)*
    \bigl(p_{n+1}\circ(\bar\beta_{n+1}*\beta_{n+1}^{-1})\bigr)^{-1}\\
    &\homotopic
    (\bar\beta_{n}*\beta_{n}^{-1})*
    \bigl(\beta_{n}*(p_{n+1}\circ\beta_{n+1}^{-1})\bigr)*
    \bigl((p_{n+1}\circ\bar\beta_{n+1})
    *(p_{n+1}\circ\beta_{n+1}^{-1})\bigr)^{-1}\\
    &\homotopic
    \bar\beta_{n}*\beta_{n}^{-1}*
    \beta_{n} * (p_{n+1}\circ\beta_{n+1}^{-1})*
    (p_{n+1}\circ\beta_{n+1})*(p_{n+1}\circ\bar\beta_{n+1}^{-1})\\
    &\homotopic
    \bar\beta_{n}*(p_{n+1}\circ\bar\beta_{n+1}^{-1})
  \end{align*}
  and so the element of $\limone_{n}\pi_{k+1}X_{n}$ represented by
  $([\beta_{n}*(p_{n+1}\circ\beta_{n+1}^{-1})])_{n\ge 0}$ equals the
  element represented by
  $([\bar\beta_{n}*(p_{n+1}\circ\bar\beta_{n+1}^{-1})])_{n\ge 0}$.
\end{proof}

\begin{prop}
  \label{prop:DefOnHomCls}
  If $\alpha,\bar\alpha\colon I^{k} \to \lim_{n}X_{n}$ represent
  elements of the kernel of $P$ and $[\alpha] = [\bar\alpha]$, then
  $\phi(\alpha) = \phi(\bar\alpha)$, and so $\phi$ defines a function
  $\phitilde\colon \Ker P \to \limone_{n}\pi_{k+1}X_{n}$.
\end{prop}

\begin{proof}
  Since $\alpha$ and $\bar\alpha$ represent the same element of
  $\pi_{k}\lim_{n}X_{n}$, there is a homotopy $\gamma\colon I^{k+1}
  \to \lim_{n}X_{n}$ such that
  \begin{displaymath}
    \gamma(p,t) =
    \begin{cases}
      \alpha(p)& \text{if $t = 0$}\\
      \bar\alpha(p)& \text{if $t = 1$, and}\\
      *& \text{if $p \in \bdry I^{k}$}\Period
    \end{cases}
  \end{displaymath}
  For every $n \ge 0$ we let $\alpha_{n} = P_{n}\circ\alpha$,
  $\bar\alpha_{n} = P_{n}\circ\bar\alpha$, and $\gamma_{n} =
  P_{n}\circ\gamma$, and we choose a nullhomotopy $\beta_{n}$ of
  $\alpha_{n}$.  For every $n \ge 0$ we can then let $\bar\beta_{n} =
  \beta_{n}*\gamma_{n}$, and $\bar\beta_{n}$ is then a nullhomotopy of
  $\bar\alpha_{n}$, and we have
  \begin{align*}
    \bar\beta_{n}*(p_{n+1}\circ\bar\beta_{n+1}^{-1})
    &= (\beta_{n}*\gamma_{n})*
    \bigl(p_{n+1}\circ(\gamma_{n+1}^{-1}*\beta_{n+1}^{-1})\bigr)\\
    &= (\beta_{n}*\gamma_{n})*
    \bigl((p_{n+1}\circ\gamma_{n+1}^{-1})
    *(p_{n+1}\circ\beta_{n+1}^{-1})\bigr)\\
    &= (\beta_{n}*\gamma_{n})*
    \bigl(\gamma_{n}^{-1}*(p_{n+1}\circ\beta_{n+1}^{-1})\bigr)\\
    &\homotopic \beta_{n}*(p_{n+1}\circ\beta_{n+1}^{-1})
  \end{align*}
  and so $\phi(\alpha) = \phi(\bar\alpha)$.
\end{proof}

\begin{prop}
  \label{prop:homomor}
  If $k \ge 1$, then $\phitilde\colon \Ker P \to
  \limone_{n}\pi_{k+1}X_{n}$ is a homomorphism.
\end{prop}

\begin{proof}
  Let $\alpha,\bar\alpha\colon I^{k} \to \lim_{n}X_{n}$ represent
  elements of $\pi_{k}\lim_{n}X_{n}$ in the kernel of $P$.  For every
  $n \ge 0$, let $\alpha_{n} = P_{n}\circ \alpha$, $\bar\alpha_{n} =
  P_{n}\circ \bar\alpha$, and choose nullhomotopies $\beta_{n}$ of
  $\alpha_{n}$ and $\bar\beta_{n}$ of $\bar\alpha_{n}$.  If
  $\beta_{n}\starhat \bar\beta_{n}\colon I^{k+1} \to X_{n}$ is defined
  as in \defref{def:altcomp}, then $\beta_{n}\starhat\bar\beta_{n}$ is
  a nullhomotopy of $\alpha_{n}*\bar\alpha_{n}$.  Thus,
  $\phitilde([\alpha][\bar\alpha])$ in level $n$ is represented by
  \begin{displaymath}
    (\beta_{n}\starhat\bar\beta_{n}) *
    \bigl(p_{n+1}\circ(\beta_{n+1}\starhat\bar\beta_{n+1})^{-1}\bigr)
    = (\beta_{n}\starhat\bar\beta_{n}) *
    \bigl((p_{n+1}\circ\beta_{n+1}) \starhat
    (p_{n+1}\circ\bar\beta_{n+1})\bigr)^{-1}
  \end{displaymath}
  which on the point $(t_{1},t_{2},\ldots,t_{k+1})$ takes the value
  \begin{alignat*}{2}
    \beta(t_{1},t_{2},\ldots,t_{k-1},2t_{k},2t_{k+1})&\quad&
    \text {if $0\le t_{k}\le\tfrac{1}{2}$ and $0\le
      t_{k+1}\le\tfrac{1}{2}$}\\ 
    \bar\beta(t_{1},t_{2},\ldots,t_{k-1},2t_{k}-1,2t_{k+1})&&
    \text {if $\tfrac{1}{2}\le t_{k}\le 1$ and $0\le
      t_{k+1}\le\tfrac{1}{2}$}\\ 
    (p_{n+1}\circ\beta)(t_{1},t_{2},\ldots,t_{k-1},2t_{k},1-(2t_{k+1}))&&
    \text {if $0\le t_{k}\le\tfrac{1}{2}$ and $\tfrac{1}{2}\le
      t_{k+1}\le1$}\\ 
    (p_{n+1}\circ\bar\beta)(t_{1},t_{2},\ldots,t_{k-1},2t_{k}-1,1-(2t_{k+1}))
    &&
    \text {if $\tfrac{1}{2}\le t_{k}\le1$ and $\tfrac{1}{2}\le
      t_{k+1}\le1$}
    \Period
  \end{alignat*}
  This is also the definition of
  \begin{displaymath}
    \bigl(\beta_{n}*(p_{n+1}\circ\beta_{n+1}^{-1})\bigr) \starhat
    \bigl(\bar\beta_{n}*(p_{n+1}\circ\bar\beta_{n+1}^{-1})\bigr)
  \end{displaymath}
  which is the level $n$ representative of
  $\phi(\alpha)\starhat\phi(\bar\alpha)$.  Since multiplication of
  elements of a homotopy group can be defined using any of the
  coordinates of the cube, we have
  $\phitilde\bigl([\alpha][\bar\alpha]\bigr) =
  \phitilde\bigl([\alpha]\bigr) \phitilde\bigl([\bar\alpha]\bigr)$.
\end{proof}

\begin{prop}
  \label{prop:surjective}
  The function $\phitilde\colon \Ker P \to \limone_{n}\pi_{k+1}X_{n}$
  is surjective for all $k \ge 0$.
\end{prop}

\begin{proof}
  Fix a value of $k \ge 0$; every element of
  $\limone_{n}\pi_{k+1}X_{n}$ can be represented by
  $\bigl([\gamma_{n}]\bigr)_{n \ge 0} \in \prod_{n\ge 0} \pi_{k+1}
  X_{n}$ where $\gamma_{n}\colon I^{k+1} \to X_{n}$ for every $n \ge
  0$ is a map taking the boundary of $I^{k+1}$ to the basepoint.  We
  will inductively define $\alpha_{n}\colon I^{k} \to X_{n}$ taking
  the boundary of $I^{k}$ to the basepoint and a nullhomotopy
  $\beta_{n}\colon I^{k+1} \to X_{n}$ of $\alpha_{n}$.  We will
  arrange it so that $p_{n+1}\circ \alpha_{n+1} = \alpha_{n}$ for all
  $n \ge 0$, so that the $\bigl(\alpha_{n}\bigr)_{n \ge 0}$ will
  define a map $\alpha\colon I^{k} \to \lim_{n}X_{n}$ taking the
  boundary of $I^{k}$ to the basepoint, and such that
  $\phitilde\bigl([\alpha]\bigr)$ is the element of
  $\limone_{n}\pi_{k+1}X_{n}$ represented by
  $\bigl([\gamma_{n}]\bigr)_{n\ge 0}$.  We begin by letting
  $\alpha_{0}\colon I^{k} \to X_{0}$ and $\beta_{0}\colon I^{k+1} \to
  X_{0}$ both be constant maps to the basepoint.

  For the inductive step, let $n \ge 0$ and assume that we've defined
  $\alpha_{n}$ and $\beta_{n}$.  The map $\gamma_{n}^{-1}*
  \beta_{n}\colon I^{k+1} \to X_{n}$ is a nullhomotopy of
  $\alpha_{n}$, and its restriction to $(I^{k}\times \{0\}) \cup
  (\bdry I^{k}\times I)$ is the constant map to the basepoint.  Thus,
  we can lift that restriction to the constant map to the basepoint of
  $X_{n+1}$.  Since there is a homeomorphism of $I^{k+1}$ to itself
  that takes $(I^{k}\times \{0\}) \cup (\bdry I^{k}\times I)$ onto
  $I^{k}\times \{0\}$, we can extend that lift to a map
  $\beta_{n+1}\colon I^{k+1} \to X_{n+1}$.  We define
  $\alpha_{n+1}(t_{1},t_{2},\ldots,t_{k}) = \beta_{n+1}(t_{1},
  t_{2},\ldots, t_{k}, 1)$, and $\beta_{n+1}$ is a nullhomotopy of
  $\alpha_{n+1}$.

  We now have $p_{n+1}\circ\alpha_{n+1} = \alpha_{n}$ for all $n \ge
  0$, and so the maps $(\alpha_{n})_{n\ge 0}$ define $\alpha\colon
  I^{k} \to \lim_{n}X_{n}$, which represents an element of
  $\pi_{k}\lim_{n}X_{n}$ in the kernel of $P$.  For every $n \ge 0$ we
  have $p_{n+1}\circ\beta_{n+1} = \gamma_{n}^{-1}* \beta_{n}$, and so
  \begin{align*}
    \beta_{n}*(p_{n+1}\circ\beta_{n+1}^{-1})
    &= \beta_{n} * (\beta_{n}^{-1} * \gamma_{n})\\
    &\homotopic \gamma_{n} \Comma
  \end{align*}
  and so $\phitilde\bigl([\alpha]\bigr)$ is represented by
  $\bigl([\gamma_{n}]\bigr)_{n\ge 0}$.  Thus, $\phitilde$ is
  surjective.
\end{proof}

\begin{prop}
  \label{prop:injective}
  The function $\phitilde\colon \Ker P \to \limone_{n}\pi_{k+1}X_{n}$
  is injective for all $k \ge 0$.
\end{prop}

\begin{proof}
  Let $\alpha,\bar\alpha\colon I^{k} \to \lim_{n}X_{n}$ represent
  elements of the kernel of $P$ such that $\phi(\alpha) =
  \phi(\bar\alpha)$ in $\limone_{n}\pi_{k+1}X_{n}$.  For each $n \ge
  0$ let $\alpha_{n} = P_{n}\circ\alpha$, let $\bar\alpha_{n} =
  P_{n}\circ\bar\alpha$, and choose nullhomotopies $\beta_{n}$ of
  $\alpha_{n}$ and $\bar\beta_{n}$ of $\bar\alpha_{n}$.  Since
  $\phi(\alpha) = \phi(\bar\alpha)$ in $\limone_{n}\pi_{k+1}X_{n}$,
  there exists an element $([g_{n}])_{n\ge 0}$ of $\prod_{n\ge 0}
  \pi_{k+1}X_{n}$ (where each map $g_{n}\colon I^{k+1} \to X_{n}$
  takes the boundary of $I^{k+1}$ to the basepoint) so that
  \begin{displaymath}
    g_{n}* \bigl(\beta_{n}*(p_{n+1}\circ\beta_{n+1}^{-1})\bigr) *
    (p_{n+1}\circ g_{n+1}^{-1})
    \homotopic
    \bar\beta_{n} * (p_{n+1}\circ\bar\beta_{n+1}^{-1})
  \end{displaymath}
  and so
  \begin{displaymath}
    g_{n}* \beta_{n}*
    \bigl(p_{n+1}\circ(\beta_{n+1}^{-1} * g_{n+1}^{-1})\bigr)
    \homotopic
    \bar\beta_{n} * (p_{n+1}\circ\bar\beta_{n+1}^{-1})
  \end{displaymath}
  for all $n \ge 0$.  For every $n \ge 0$ let $\gamma_{n} =
  g_{n}*\beta_{n}$; we then have
  \begin{displaymath}
    \gamma_{n}*(p_{n+1}\circ\gamma_{n+1}^{-1}) \homotopic
    \bar\beta_{n}*(p_{n+1}\circ\bar\beta_{n+1}^{-1})
  \end{displaymath}
  and so
  \begin{displaymath}
    \bar\beta_{n}^{-1}*\gamma_{n} \homotopic
    (p_{n+1}\circ\bar\beta_{n+1}^{-1}) * (p_{n+1}\circ\gamma_{n+1})
    = p_{n+1}\circ(\bar\beta_{n+1}^{-1} * \gamma_{n+1}) \Period
  \end{displaymath}
  Thus, for every $n \ge 0$, we have
  \begin{itemize}
  \item $\bar\beta_{n}^{-1}* \gamma_{n}$ is a homotopy in $X_{n}$ from
    $\bar\alpha_{n}$ to $\alpha_{n}$, and
  \item $p_{n+1}\circ(\bar\beta_{n+1}^{-1}*\gamma_{n+1}) \homotopic
    \bar\beta_{n}^{-1}*\gamma_{n}$.
  \end{itemize}
  We will now inductively replace each homotopy
  $\bar\beta_{n}^{-1}*\gamma_{n}$ with a homotopic homotopy
  $\delta_{n}$ so that $p_{n+1}\circ\delta_{n+1} = \delta_{n}$; the
  $(\delta_{n})_{n\ge 0}$ will then define a homotopy in
  $\lim_{n}X_{n}$ from $\bar \alpha$ to $\alpha$.  We begin by letting
  $\delta_{0} = \bar\beta_{0}^{-1} * \gamma_{0}$.

  If $n \ge 0$ and we've defined $\delta_{n}$, then
  $p_{n+1}\circ(\bar\beta_{n+1}^{-1}*\gamma_{n+1})$ is homotopic to
  $\delta_{n}$, i.e., there is a homotopy $H\colon I^{k+2} \to X_{n}$
  such that
  \begin{align*}
    H(p,t) &= \bigl(p_{n+1}\circ (\bar\beta_{n+1}^{-1} *
    \gamma_{n+1})\bigr) (p)&&
    \text{for $(p,t) \in (I^{k+1}\times \{0\}) \cup
      (\bdry I^{k+1}\times I)$}\\
    H(p,1) &= \delta_{n}(p)&&\text{for $p \in I^{k+1}$}
  \end{align*}
  If we define $H'\colon (I^{k+1}\times \{0\}) \cup (\bdry I^{k+1}
  \times I) \to X_{n+1}$ by $H'(p,t) = (\bar\beta_{n+1}^{-1} *
  \gamma_{n+1})(p)$, then $H'$ is a lift of the restriction of $H$ to
  $(I^{k+1}\times \{0\}) \cup (\bdry I^{k+1}\times I)$.  There is a
  homeomorphism of $I^{k+2}$ to itself that takes $(I^{k+1}\times
  \{0\}) \cup (\bdry I^{k+1}\times I)$ onto $I^{k+1} \times \{0\}$,
  and so we can extend $H'$ to a lift $H'\colon I^{k+2} \to X_{n+1}$
  of $H$.  We define $\delta_{n+1}\colon I^{k+1} \to X_{n+1}$ by
  letting $\delta_{n+1}(p) = H'(p,1)$, and $\delta_{n+1}$ is a
  homotopy from $\bar\alpha_{n+1}$ to $\alpha_{n+1}$, homotopic to
  $\bar\beta_{n+1}^{-1}* \gamma_{n+1}$, such that $p_{n+1}\circ
  \delta_{n+1} = \delta_{n}$.  This completes the induction, and so
  the $(\delta_{n})_{n\ge 0}$ define a homotopy $\delta$ in
  $\lim_{n}X_{n}$ from $\bar\alpha$ to $\alpha$, and so $\phitilde$ is
  injective.
\end{proof}

\begin{prop}
  \label{prop:most}
  For every $k \ge 0$ the natural map $P\colon \pi_{k}\lim_{n} X_{n}
  \to \lim_{n}\pi_{k} X_{n}$ of \thmref{thm:HomGrps} is surjective and
  its kernel is naturally isomorphic to $\limone_{n}\pi_{k+1}X_{n}$.
\end{prop}

\begin{proof}
  \propref{prop:surj} shows that $P$ is surjective,
  \propref{prop:welldefined} and \propref{prop:DefOnHomCls} define a
  natural map from the kernel of $P$ to $\limone_{n}\pi_{k+1}X_{n}$,
  \propref{prop:homomor} shows that if $k \ge 1$ then that natural map
  is a homomorphism, and \propref{prop:surjective} and
  \propref{prop:injective} show that it is an isomorphism.
\end{proof}

\begin{prop}
  \label{prop:pi0Hspace}
  If $k = 0$, all of the spaces $X_{n}$ are H-spaces, and all of the
  maps $p_{n}$ are H-maps, then $\phitilde\colon \Ker P \to
  \limone_{n}\pi_{1}X_{n}$ is a homomorphism.
\end{prop}

\begin{proof}
  If $W$ is a space, $X$ is an H-space, and $f,g\colon W \to X$ are
  maps, then we denote the product of $f$ and $g$ defined using the
  H-space structure of $X$ as $f\startilde g$.

  Let $\alpha,\bar\alpha\colon I^{0} \to \lim_{n}X_{n}$ be maps such
  that the elements $[\alpha]$ and $[\bar\alpha]$ of
  $\pi_{0}\lim_{n}X_{n}$ are in the kernel of $P$.  For every $n \ge
  0$ we let $\alpha_{n} = P_{n}\circ\alpha$, let $\bar\alpha_{n} =
  P_{n}\circ\bar\alpha$, and choose nullhomotopies $\beta_{n}\colon I
  \to X_{n}$ of $\alpha_{n}$ and $\bar\beta_{n}\colon I \to X_{n}$ of
  $\bar\alpha$.

  The multiplication in both $\pi_{0}$ and $\pi_{1}$ of an H-space can
  be defined using the H-space product, and so the product of
  $[\alpha]$ and $[\bar\alpha]$ in $\pi_{0}\lim_{n}X_{n}$ is $[\alpha
  \startilde \bar\alpha]$.  If for all $n \ge 0$ we let
  $(\alpha\startilde \bar\alpha)_{n} = P_{n}\circ (\alpha\startilde
  \bar\alpha)$, then $\beta_{n}\startilde \bar\beta_{n}$ is a
  nullhomotopy of $(\alpha\startilde \bar\alpha)_{n} = \alpha_{n}
  \startilde \bar\alpha_{n}$.  Thus, $\phi(\alpha \startilde
  \bar\alpha)$ is represented by
  \begin{displaymath}
    \bigl((\beta_{n}\startilde \bar\beta_{n}) *
    \bigl(p_{n+1}\circ(\beta_{n+1} \startilde
    \bar\beta_{n+1})^{-1}\bigr)\bigr)_{n \ge 0}
  \end{displaymath}
  while $\phi(\alpha) \startilde \phi(\bar\alpha)$ is represented by
  \begin{displaymath}
    \bigl(\bigl(\beta_{n}*(p_{n+1}\circ\beta_{n+1}^{-1})\bigr) \startilde
    \bigl(\bar\beta_{n} * (p_{n+1}\circ\bar\beta_{n+1}^{-1})\bigr)
    \bigr)_{n\ge 0} \Period
  \end{displaymath}
  Since those two maps are equal, $\phi(\alpha\startilde \bar\alpha) =
  \phi(\alpha)\startilde \phi(\bar\alpha)$, and so
  $\phitilde\bigl([\alpha][\bar\alpha]\bigr) =
  \phitilde\bigl([\alpha]\bigr) \phitilde\bigl([\bar\alpha]\bigr)$.
\end{proof}

\begin{proof}[Proof of \thmref{thm:HomGrps}]
  This follows from \propref{prop:surj}, \propref{prop:welldefined},
  \propref{prop:DefOnHomCls}, \propref{prop:homomor},
  \propref{prop:surjective}, \propref{prop:injective}, and
  \propref{prop:pi0Hspace}.
\end{proof}




\begin{bibdiv} 
  \begin{biblist}

\bib{YM}{book}{
   author={Bousfield, A. K.},
   author={Kan, D. M.},
   title={Homotopy limits, completions and localizations},
   series={Lecture Notes in Mathematics, Vol. 304},
   publisher={Springer-Verlag, Berlin-New York},
   date={1972},
   pages={v+348},
}

\bib{Cohen-aarhus}{article}{
   author={Cohen, Joel M.},
   title={Homotopy groups of inverse limits},
   conference={
      title={Proceedings of the Advanced Study Institute on Algebraic
      Topology },
      address={Aarhus Univ., Aarhus},
      date={1970},
   },
   book={
      publisher={Mat. Inst., Aarhus Univ., Aarhus},
   },
   date={1970},
   pages={29--43. Various Publ. Ser., No. 13},
}

\bib{Cohen}{article}{
   author={Cohen, Joel M.},
   title={The homotopy groups of inverse limits},
   journal={Proc. London Math. Soc. (3)},
   volume={27},
   date={1973},
   pages={159--177},
   issn={0024-6115},
}

\bib{RHT}{article}{
   author={Quillen, Daniel},
   title={Rational homotopy theory},
   journal={Ann. of Math. (2)},
   volume={90},
   date={1969},
   pages={205--295},
   issn={0003-486X},
}

\bib{Gray}{book}{
   author={Gray, Brayton I.},
   title={OPERATIONS AND A PROBLEM OF HELLER},
   note={Thesis (Ph.D.)--The University of Chicago},
   publisher={ProQuest LLC, Ann Arbor, MI},
   date={1965},
   pages={(no paging)},
}

\bib{Vogt}{article}{
   author={Vogt, R. M.},
   title={On the dual of a lemma of Milnor},
   conference={
      title={Proceedings of the Advanced Study Institute on Algebraic
      Topology (1970), Vol. III},
   },
   book={
      publisher={Mat. Inst., Aarhus Univ., Aarhus},
   },
   date={1970},
   pages={632--648. Various Publ. Ser., No. 13},
}

  \end{biblist}
\end{bibdiv}
\end{document}